\documentclass[smallextended,envcountsect]{svjour3}
\smartqed

\usepackage{graphicx}
\usepackage{latexsym,amsmath,amsfonts,amscd,amssymb,verbatim,hyperref}
\usepackage{color}
\usepackage{cite}
\usepackage[shortlabels]{enumitem}

\def\disp{\displaystyle}

\def\tto{\;{\lower 1pt \hbox{$\rightarrow$}}\kern -10pt
\hbox{\raise 2pt \hbox{$\rightarrow$}}\;}

\def\Bar{\overline}
\def\ra{\rangle}
\def\la{\langle}

\def\epsilon{\varepsilon}

\def\R{\Bbb R}

\def\ox{\bar{x}}

\def\oy{\bar{y}}

\def\gph{\mbox{\rm gph}\,}

\def\epi{\mbox{\rm epi}\,}

\def\dom{\mbox{\rm dom}\,}

\def\cone{\mbox{\rm cone}}

\def\oR{\Bar{\R}}

\def\oR{\Bar{\R}}

\begin{document}

\title{On Differential Stability of a Class of Convex Optimization Problems}

\author{Nguyen~Dong~Yen \and Duong~Thi~Viet~An \and Vu~Thi~Huong \and Nguyen~Ngoc~Luan}

\institute{
	Nguyen Dong Yen, Corresponding author \at
	Institute of Mathematics, Vietnam Academy of	Science and Technology\\
	18 Hoang Quoc Viet, Hanoi 10307, Vietnam\\
	ndyen@math.ac.vn
	\and
	Duong Thi Viet An \at Department of Mathematics and Informatics, Thai Nguyen University of Sciences, Thai Nguyen, 250000, Vietnam\\
     andtv@tnus.edu.vn
     \and
     Vu Thi Huong \at
     Institute of Mathematics, Vietnam Academy of	Science and Technology\\
     18 Hoang Quoc Viet, Hanoi 10307, Vietnam\\
     vthuong@math.ac.vn
	\and
	Nguyen Ngoc Luan \at Department of Mathematics and Informatics, Hanoi National University of Education, 136 Xuan Thuy, Hanoi, Vietnam\\
	luannn@hnue.edu.vn
}

\date{Received: date / Accepted: date}
 %The correct dates will be entered by the editor.

 \titlerunning{Differential Stability of a Class of Convex Optimization Problems}

 \authorrunning{N. D. Yen, D. T. V. An, V. T. Huong, N. N. Luan}

 \maketitle

%\begin{center}
%	(Dedicated to Professor... on the occasion of...)
% \end{center}

\begin{abstract} The recent results of An, Luan, and Yen [Differential stability in convex optimization via generalized polyhedrality. Vietnam J. Math.  https://-doi.org/10.1007/s10013-024-00721-y] on differential stability of  parametric optimization problems described by proper generalized polyhedral convex functions and generalized polyhedral convex set-valued maps are analyzed, developed, and sharpened in this paper. Namely, keeping the Hausdorff locally convex topological vector spaces setting, we clarify the relationships between the upper estimates and lower estimates for the subdifferential and the singular subdifferential of the optimal value function. As shown by an example, the lower estimates can be strict. But, surprisingly, each upper estimate is an equality. Thus, exact formulas for the subdifferential and the singular subdifferential under consideration are obtained. In addition, it is proved that  each subdifferential upper estimate coincides with the corresponding lower estimate if either the objective function or the constraint set-valued map is polyhedral~convex.

\end{abstract}

\section{Introduction}

Differential stability of convex optimization problems has been studied by many authors. Among the first results, we would like to refer to Theorem~29.1 from the classical book of Rockafellar~\cite{R_1970}, where the subdifferential of the perturbation function of a convex program associated with a convex bifunction between two Euclidean spaces was computed via the set of the Kuhn-Tucker vectors. Since convex bifunctions can be considered as generalizations of set-valued maps (see~\cite[p.~292]{R_1970}), where instead of the indicator function of each value of the map one may have an extended-real-valued function, this theorem yields a formula for the subdifferential  of the optimal value function of certain convex programs under inclusion constraints.  

Note that differential stability of convex optimization problems under inclusion constraints has received much attention from researchers. A formula for computing the subdifferential of the optimal
value function of an optimization problem in Hilbert spaces, where the perturbations are canonical and the objective function is unperturbed, was obtained by Aubin~\cite[Problem~35~-~Subdifferentials of Marginal Functions, p.~335]{Aubin_1998} under a regularity condition. An analysis and extensions of the result were given by An and Yen~\cite{AnYen2015} for optimization problems in Hausdorff locally convex topological vector spaces. More comprehensive results in this direction have been established by Mordukhovich et al.~\cite[Proposition~7.1 and Theorem~7.2]{bmnrt2017} (see also Theorem~4.56 and the commentaries on p.~309 in the book by  Mordukhovich and Nam~\cite{bmn2022}).  

Various results on differential stability of convex optimization problems under inclusion constraints and their applications to optimal control of discrete or continuous dynamical systems can be found in the books by Mordukhovich and Nam~\cite[pp.~106--108]{bmn} and~\cite[pp.~286--288]{bmn2022}, and the papers ~\cite{An_Huong_Xu_2022,An_Jourani,An_Toan_2018,An_Yao_2016,AnYaoYen2020,AnYen2018,Mahmudov_2018,bmnrt2017,Toan_Thuy_2021,Thuy_Toan_2016}.

Interesting differential stability results for convex optimization problems via $\varepsilon$-subdifferential and infimal convolution of convex functions, are available  in~\cite{An_Gutierrez_2021,An_Kobis_Tuyen_2020,An_Yao_2019,Toan_2022}. 
In~\cite{LuanKimYen_2022}, some results of~\cite{AnYen2015} have been effectively used in the study of differential stability of parametric conic linear programming problems.

Recently, the problem of computing or estimating the subdifferential of the optimal value function of a parametric optimization problem described by a proper generalized polyhedral convex function and a generalized polyhedral convex set-valued map has been considered by An et al.~\cite{AnLuanYen_2023}. Using the Hausdorff locally convex topological vector spaces setting, the authors have obtained upper estimates and lower estimates for the subdifferential and the singular subdifferential of the optimal value function at a given parameter.

The interested reader is referred to the paper by Luan and Yao~\cite{LuanYao_2019} for a systematic investigation on the solution existence, optimality conditions, and duality theorems for generalized polyhedral convex optimization problems.

Our purpose in the present paper is to solve Question~3 in~\cite{AnLuanYen_2023}, which asks for examples justifying that each subdifferential estimate given there can be strict. First, we will have a closer look at the subdifferential upper and lower estimates provided in that paper. Then, we will construct a nontrivial example (see the proof of Theorem~\ref{thm_solution1} below) to show that the lower estimates can be strict. Next, we will prove that each upper estimate given  in~\cite{AnLuanYen_2023} is an equality. By doing so, we get exact formulas for the subdifferential and the singular subdifferential of the optimal value function at a given parameter. Finally, for the case where either the objective function or the constraint set-valued mapping is polyhedral convex, it will be shown that each upper estimate coincides with the corresponding lower estimate.

The organization of this paper is as follows. Basic definitions and some preliminaries are collected in Section~\ref{Sect_2}. Sections~\ref{sect_3} and~\ref{sect_4} address the above-mentioned open question from~\cite{AnLuanYen_2023}. Section~\ref{sect_5} considers differential stability under a polyhedral convexity assumption, and Section~\ref{sect_6} gives some concluding remarks. 

\section{Basic Definitions and Preliminaries}\label{Sect_2}

Let $X$ and $Y$ be real Hausdorff locally convex topological vector spaces. Denote the dual spaces of $X$ and $Y$, respectively, by $X^*$ and~$Y^*$. For $x^*\in X^*$ and $x\in X$, the value $x^*$ at $x$ is abbreviated to $\la x^*,x\ra$. Let $X^*$ and $Y^*$ be equipped with the weak$^*$ topology. Details about the latter concept can be found in~\cite[Definition~1.107 and Subsection~1.2.2]{bmn2022} and~\cite[pp.~67--68]{Rudin_1991}. For a subset $C\subset X$, by $\overline{C}$ we denote the closure of  $C$ in the topology of $X$. Similarly, given a subset  $D\subset X^*$, we let $\overline{D}$ stand for the closure of  $D$ in the weak$^*$ topology of $X^*$. The cone generated by $C$ is defined by $\cone\, C=\{tx \mid t\geq 0, x\in C\}$. The extended real line is the set $\oR:=\R\cup \{\pm \infty\}$. 

One says that a  set $C$ in $X$ is {\em convex} if $(1-t) x+tu\in C$ for all $x, u \in C$ and $t\in [0,1]$.
The {\em normal cone} to a convex set $C$ at $\ox\in C$ is defined by
\begin{equation*}\label{def_normal_cone}
N(\ox;C)=\big\{x^*\in X^*\; \big |\; \la x^*, x-\ox\ra\leq 0\ \; \mbox{\rm for all }\; x\in C\big\}.
\end{equation*} For any  $\ox\notin C$, we put $N(\ox;C)=\emptyset$. 

The \textit{indicator function} $\delta(.;C)$ of a subset $C\subset X$ is given by $\delta(x;C)=0$ for $x\in C$ and $\delta(x;C)=+\infty$ for $x\in X\setminus C$.

Given a function $f\colon X\to \oR$, we define the {\em epigraph} and the {\em effective domain} of $f$ respectively by
\begin{align*}
   & \epi f=\big\{(x, \lambda)\in X\times \R\; \big |\; f(x)\leq \lambda\big\}\ \; \mbox{\rm and }\;
   \dom f=\big\{x\in X\; \big |\; f(x)<+\infty\big\}.
\end{align*}
If $\epi f$ is a convex set in $X\times \R$, one says that $f$ is a {\em convex function}. If $\dom f\neq\emptyset$ and $f(x)\neq -\infty$ for all $x\in X$,  then $f$ is called a {\em proper function}. 

Let $f: X\rightarrow \overline{\Bbb{R}}$ be a convex function and $\bar x\in X$ be such that $f(\bar x)\in\mathbb R$.  The {\it subdifferential} of $f$ at $\bar x$ is defined by \begin{align*}\label{subdifferential_convex_analysis}\partial f(\bar x)=\big\{x^* \in X^* \mid \langle x^*, x- \bar x \rangle \le f(x)-f(\bar x), \ \forall x \in X\big\}.\end{align*} 
The {\it singular subdifferential} of $f$ at $\bar x$ is the set 
\begin{eqnarray*}\label{singular_subdff}\partial^\infty f(\bar x):=\big\{x^* \in X^* \mid (x^*,0)\in N ( (\bar x, f(\bar x)); {\rm{epi}}\, f)\big\}.\end{eqnarray*} 
We put $\partial f(\bar x)=\emptyset$ and $\partial^\infty f(\bar x)=\emptyset$ if either $\bar x\notin {\rm dom}\,f$ or  $f(\bar x)=-\infty$. 

\begin{proposition} {\rm (See~\cite[Proposition~4.2]{AnYen2015})}\label{prop4.2_An_Yen}
	If $f: X \rightarrow \overline{\Bbb{R}}$ is a convex function, then
	$$ \partial^\infty f(x)=N(x;{\rm{dom}}\, f)=\partial\delta(.;{\rm{dom}}\, f)(x)\quad\, \forall x\in X.$$ 
\end{proposition}

It follows from Proposition~\ref{prop4.2_An_Yen} that $\partial^\infty f(\bar x)$ is a weakly$^*$-closed convex cone. Geometrically, for a point $\bar x\in X$ with $f(\bar x)\in\mathbb R$, the lager is the cone, the sharper is the set ${\rm dom}\, f$ around $\bar x$.

The {\em graph} and the {\em effective domain} of a set-valued map $F\colon X\tto Y$ are defined 
respectively by
\begin{align*}
	& \gph F=\big\{(x, y)\in X\times Y\; \big |\; y\in F(x)\big\}\ \; {\rm and}\ \; \dom F=\big\{x\in X\; \big |\; F(x)\neq\emptyset\big\}.
\end{align*}
If $\gph F$ is a convex set, then we say that $F$ is a {\em convex set-valued map}. The {\em coderivative} of a convex set-valued map $F: X \rightrightarrows Y$ at $(\ox, \oy)\in \gph F$ is defined~by
\begin{equation*}\label{def_coderivative}
	D^*F(\ox, \oy)(y^*)=\big\{x^*\in X^*\; \big |\; (x^*, -y^*)\in N\big((\ox, \oy); \gph F\big)\big\}, \ \; y^*\in Y^*.
\end{equation*} 
If $(\bar x, \bar y) \notin {\rm gph}\, F$, then we put $D^* F(\bar x, \bar y)(y^*)=\emptyset$ for any $y^* \in Y^*$.

Now, let us recall the notion of generalized polyhedral convex set from~\cite{Bonnans_Shapiro_2000}, as well as the notion of generalized polyhedral convex function and several results from~\cite{LuanYaoYen_2018}.

A subset $D\subset X$ is said to be a \textit{generalized polyhedral convex set} if there exist $x_i^*\in X^*$, $\alpha_i \in \mathbb{R},$ $ i=1,2,\dots,p$, and a closed affine subspace $L \subset X$, such~that
	\begin{align}\label{generalized_polyhedra_convex}
		D=\{x\in X \mid x\in L, \ \langle x_i^*, x \rangle \le \alpha_i , \ i=1,2,\dots,p\}.
	\end{align} 
If $D$ can be represented in the form of \eqref{generalized_polyhedra_convex} with $L=X$, then we say that it is a \textit{polyhedral convex set}.

Clearly, every generalized polyhedral convex set is a closed set. Note also that if $X$ is finite-dimensional, a subset $D\subset X$ is a generalized polyhedral convex set if and only if it is a polyhedral convex set.

A function $f:X\to\overline{\mathbb{R}}$ is called \textit{generalized polyhedral convex} (resp., \textit{polyhedral convex}) if its epigraph is a generalized polyhedral convex set (resp., a polyhedral convex set) in $X\times\mathbb R$.

Complete characterizations of a generalized polyhedral convex function (resp., a polyhedral convex function) in the form of the maximum of a finite family of continuous affine functions over a certain generalized polyhedral convex set (resp., a polyhedral convex set) are given in the following theorem.

\begin{theorem}\label{thm_gpcf}	{\rm (See~\cite[Theorem~3.2]{LuanYaoYen_2018})} Suppose that $f:X\to\overline{\mathbb{R}}$ is a proper function. Then $f$ is generalized polyhedral convex (resp., polyhedral convex) if and only if ${\rm dom}f$ is a generalized polyhedral convex set (resp., a polyhedral convex set) in $X$ and there exist $v_k^* \in X^*$, $\beta_k \in \mathbb{R}$, for $k=1,\dots,m$, such that
	\begin{equation*}\label{eq_rep_gcpf}
		f(x)=\begin{cases}
			\max \big\{ \langle v_k^*, x \rangle + \beta_k \mid k=1,\dots,m  \big\} &\text{if } x \in {\rm dom}f,\\
			+\infty & \text{if } x \notin {\rm dom}f.
		\end{cases}
	\end{equation*}	  
\end{theorem}

The specific structure of  generalized polyhedral convex functions allows one to have the next subdifferential sum rule without any regularity assumption.

\begin{theorem}\label{thm_sum_rule} {\rm (See~\cite[Theorem~4.16]{LuanYaoYen_2018})}
	Let $f_1,f_2, \dots, f_m$ be proper generalized polyhedral convex functions on $X$. Then, for any $x \in \bigcap\limits_{i=1}^m {\rm dom}f_i$, 
	\begin{equation*}\label{eq_sum_subdifferentials_gpcf}
		\partial (f_1+f_2+\dots+f_m)(x)=\overline{\partial f_1(x)+ \partial f_2(x)+\dots+ \partial f_m(x)}.
	\end{equation*}
\end{theorem}

\begin{theorem}\label{thm_sum_rule_2} {\rm (See~\cite[Theorem~4.17]{LuanYaoYen_2018})}
	Suppose $f_1$ is a proper polyhedral convex function on $X$ and $f_2$ is a proper generalized polyhedral convex function on $X$. Then, for any $x \in \big({\rm dom}f_1\big)\cap \big({\rm dom}f_2\big)$, 
	\begin{equation*}\label{eq_sum_subdifferentials_gpcf}
		\partial (f_1+f_2)(x)=\partial f_1(x)+ \partial f_2(x).
	\end{equation*}
\end{theorem}

Given a set-valued map $G: X \rightrightarrows Y$ and a function $\varphi: X \times Y \rightarrow \overline{\Bbb{R}}$, we consider the optimization problem depending on a parameter $x$:
\begin{align}\label{math_program}
	\tag{$P_x$} \quad \quad \quad \min\{\varphi(x,y)\mid y \in G(x)\}.
\end{align} The \textit{optimal value function}
of~\eqref{math_program} is the function $\mu: X \rightarrow \overline{\Bbb{R}}$ with
\begin{align}\label{marginalfunction}
	\mu(x):= \inf \left\{\varphi (x,y)\mid y \in G(x)\right\}.
\end{align}
By the convention $\inf \emptyset =+\infty$, we have $\mu(x)=+\infty$ for any $x \notin {\rm{dom}}\, G.$ The \textit{solution map}  $M:X\rightrightarrows Y $ of the problem~\eqref{math_program} is defined by setting \begin{align}\label{solution_map}
	M(x)=\{y \in G(x)\mid \mu(x)= \varphi (x,y)\},\quad \forall x\in X.
\end{align}

If $G$ is a convex set-valued map and if $\varphi$ is a convex function, then~\eqref{math_program} is called a {\it convex optimization problem}. 

\medskip
When the set-valued map $G$ and the function $\varphi$ are both generalized polyhedral convex,~\eqref{math_program} is a \textit{parametric generalized  polyhedral convex optimization problem} in the sense of Luan and Yao~\cite[p.~792]{LuanYao_2019}. In that case, $\mu$ is a convex function. By~\cite[Theorem~4.7]{LuanNamYen_2023}, we know that the function $\mu$ is generalized polyhedral convex
if and only if it is lower semicontinuous on $X$.

\medskip
The following inner and upper estimates for the subdifferential of $\mu$ at a given parameter point $\bar x$ have been obtained recently.

\begin{theorem}\label{thm_subdiff} {\rm (See~\cite[Theorem~3.1]{AnLuanYen_2023})} Let $\varphi: X \times Y \rightarrow \overline{\Bbb{R}}$ be a proper generalized polyhedral convex function and $G: X \rightrightarrows Y$ a generalized polyhedral convex set-valued map. Then for any $\bar x\in X$ with $\mu(\bar x)\in\mathbb R$, and for any $\bar y \in M(\bar x)$, one has
	\begin{equation}\label{subdiff_outer}
		\partial\mu(\bar x) \subset {\rm pr}_{X^*}\left[\overline{\partial  \varphi(\bar x, \bar y) + N( (\bar x, \bar y); {\rm {gph}}\,G)}\cap \big(X^*\times\{0\}\big)\right],
	\end{equation} where ${\rm pr}_{X^*}: X^*\times Y^*\to X^*,\ {\rm pr}_{X^*}(x^*,y^*):=x^*,$ is the natural projection from $X^*\times Y^*$ to $X^*$, and
	\begin{equation}\label{subdiff_inner}
		\partial \mu(\bar x) \supset \overline{\bigcup\limits_{(x^*,y^*) \in \partial \varphi(\bar x, \bar y)}   \big\{x^* + D^*G( \bar x, \bar y)(y^*) \big\}}.
	\end{equation}
\end{theorem}

\medskip
 For the singular subdifferential of $\mu$  at a given parameter point $\bar x$, the next inner and upper estimates are available.

	\begin{theorem}\label{thm_sing} {\rm (See~\cite[Theorem~3.2]{AnLuanYen_2023})} Under the assumptions of Theorem~\ref{thm_subdiff}, for any $\bar x \in {\rm dom}\, \mu$, with $\mu(\bar x)\neq -\infty$, and for any $\bar y \in M(\bar x)$, it holds that
	\begin{equation}\label{sing_outer}
		\partial^\infty\mu(\bar x) \subset {\rm pr}_{X^*}\left[\overline{\partial^\infty \varphi(\bar x, \bar y) + N( (\bar x, \bar y); {\rm {gph}}\,G)}\cap \big(X^*\times\{0\}\big)\right]
	\end{equation} and
	\begin{equation}\label{sing_inner}
		\partial^\infty \mu(\bar x) \supset \overline{\bigcup\limits_{(x^*,y^*) \in \partial^\infty \varphi(\bar x, \bar y)}   \big\{x^* + D^*G( \bar x, \bar y)(y^*) \big\}}.
	\end{equation}
\end{theorem}

By constructing suitable examples, it has been shown in~\cite[Section~4]{AnLuanYen_2023} that all the estimates given in Theorems~\ref{thm_subdiff} and~\ref{thm_sing} are sharp, i.e., each of them can hold as an equality. In that context, the following open question arises naturally (see~\cite[Question~3]{AnLuanYen_2023}). 

\smallskip
\textbf{Question~1.} \textit{Are there some examples showing that each estimate provided by Theorems~\ref{thm_subdiff} and~\ref{thm_sing} can be strict?}

\medskip
In the next two sections, we will provide a solution to this question and, moreover, study the estimates~\eqref{subdiff_outer}--\eqref{sing_inner} in detail. To do so, we denote by $A$ (resp., $B$) the set on the right-hand side of the inclusion~\eqref{subdiff_inner} (resp.,  of the inclusion~\eqref{subdiff_outer}). That is,
\begin{equation}\label{A}
	A=\overline{\bigcup\limits_{(x^*,y^*) \in \partial \varphi(\bar x, \bar y)}   \big\{x^* + D^*G( \bar x, \bar y)(y^*) \big\}}
\end{equation}
and 
\begin{equation}\label{B}
	B={\rm pr}_{X^*}\left[\overline{\partial\varphi(\bar x, \bar y) + N( (\bar x, \bar y); {\rm {gph}}\,G)}\cap \big(X^*\times\{0\}\big)\right].
\end{equation}
Similarly, denote by $A^\infty$ (resp., $B^\infty$) the set on the right-hand side of the inclusion~\eqref{sing_inner} (resp.,  of the inclusion~\eqref{sing_outer}). That is,
\begin{equation}\label{A-infty}
	A^\infty=\overline{\bigcup\limits_{(x^*,y^*) \in \partial^\infty \varphi(\bar x, \bar y)}   \big\{x^* + D^*G( \bar x, \bar y)(y^*) \big\}}
\end{equation}
and 
\begin{equation}\label{B-infty}
	B^\infty={\rm pr}_{X^*}\left[\overline{\partial^\infty \varphi(\bar x, \bar y) + N( (\bar x, \bar y); {\rm {gph}}\,G)}\cap \big(X^*\times\{0\}\big)\right].
\end{equation} With these notations, the estimates~\eqref{subdiff_outer} and~\eqref{subdiff_inner} can be rewritten as
\begin{equation}\label{estimates_1}
	A\subset \partial\mu(\bar x) \subset B,
\end{equation} while the estimates~\eqref{sing_outer} and~\eqref{sing_inner} read 
\begin{equation}\label{estimates_2}
	A^\infty\subset \partial^\infty\mu(\bar x) \subset B^\infty.
\end{equation}

\section{Solution to Question~1}\label{sect_3}

Let $A, B, A^\infty, B^\infty$ be defined as in~\eqref{A}--\eqref{B-infty}. Consider the following subsets of $A, B, A^\infty,$ and $B^\infty$:
\begin{equation}\label{A_0}
	A_0=\bigcup\limits_{(x^*,y^*) \in \partial \varphi(\bar x, \bar y)}   \big\{x^* + D^*G( \bar x, \bar y)(y^*) \big\},
\end{equation}
\begin{equation}\label{B_0}
	B_0={\rm pr}_{X^*}\left[\partial\varphi(\bar x, \bar y) + N( (\bar x, \bar y); {\rm {gph}}\,G)\cap \big(X^*\times\{0\}\big)\right],
\end{equation}
\begin{equation}\label{A-infty_0}
	A^\infty_0=\bigcup\limits_{(x^*,y^*)\in \partial^\infty \varphi(\bar x, \bar y)}   \big\{x^* + D^*G( \bar x, \bar y)(y^*) \big\},
\end{equation}
and 
\begin{equation}\label{B-infty_0}
	B^\infty_0={\rm pr}_{X^*}\left[\partial^\infty \varphi(\bar x, \bar y) + N( (\bar x, \bar y); {\rm {gph}}\,G)\cap \big(X^*\times\{0\}\big)\right].
\end{equation}
\begin{proposition}\label{prop_equalities} The equalities $B_0=A_0$ and $B^\infty_0=A^\infty_0$ hold.
\end{proposition}
\begin{proof} Indeed, we have
	\begin{equation*}
		\begin{array}{rcl}
			u^*\in B_0 & \Leftrightarrow & \begin{cases}
				\exists (x^*,y^*)\in \partial \varphi(\bar x, \bar y),\ \exists (\xi^*,\eta^*)\in N((\bar x, \bar y); {\rm gph} G)\\
				{\rm with}\ y^*+\eta^*=0\ {\rm such\ that}\ x^*+\xi^*=u^*
			\end{cases}\\
			& \Leftrightarrow & \begin{cases}
				\exists (x^*,y^*)\in \partial \varphi(\bar x, \bar y)\ {\rm and}\ \xi^*\in D^*G( \bar x, \bar y)(y^*)\\
				{\rm such\ that}\ x^*+\xi^*=u^*
			\end{cases}\\
			& \Leftrightarrow & u^*\in \bigcup\limits_{(x^*,y^*) \in \partial \varphi(\bar x, \bar y)}   \big\{x^* + D^*G( \bar x, \bar y)(y^*) \big\}\\
			& \Leftrightarrow & u^*\in A_0.
		\end{array} 
	\end{equation*} Hence, $B_0=A_0$. Similarly, we have
\begin{equation*}
	\begin{array}{rcl}
		u^*\in B^\infty_0 & \Leftrightarrow & \begin{cases}
			\exists (x^*,y^*)\in \partial^\infty\varphi(\bar x, \bar y),\ \exists (\xi^*,\eta^*)\in N((\bar x, \bar y);{\rm gph} G)\\
			{\rm with}\ y^*+\eta^*=0\ {\rm such\ that}\ x^*+\xi^*=u^*
		\end{cases}\\
		& \Leftrightarrow & \begin{cases}
			\exists (x^*,y^*)\in \partial^\infty\varphi(\bar x, \bar y)\ {\rm and}\ \xi^*\in D^*G( \bar x, \bar y)(y^*)\\
			{\rm such\ that}\ x^*+\xi^*=u^*
		\end{cases}\\
		& \Leftrightarrow & u^*\in \bigcup\limits_{(x^*,y^*) \in \partial^\infty\varphi(\bar x, \bar y)}   \big\{x^* + D^*G( \bar x, \bar y)(y^*) \big\}\\
		& \Leftrightarrow & u^*\in A^\infty_0;
	\end{array} 
\end{equation*} thus, $B^\infty_0=A^\infty_0$.
$\hfill\Box$	
	\end{proof}

\begin{proposition}The sets $A_0, B_0,A^\infty_0$, and $B^\infty_0$ are convex.
\end{proposition}
\begin{proof} Thanks to Proposition~\ref{prop_equalities}, it suffices to show that sets $B_0$ and $B^\infty_0$ are convex. But this fact is immediate from~\eqref{B_0},~\eqref{B-infty_0}, and the convexity of the subdifferential $\partial\varphi(\bar x, \bar y)$, the singular subdifferential $\partial^\infty \varphi(\bar x, \bar y)$, and the normal cone $N((\bar x, \bar y); {\rm gph} G)$. 
$\hfill\Box$	
\end{proof}

The forthcoming theorem shows that the inclusions $A\subset B$ and $A^\infty\subset B^\infty$ hold under much weaker conditions than those formulated in Theorems~\ref{thm_gpcf}	and~\ref{thm_sing}. Of course, these weaker conditions are not enough to have the inner and outer estimates for the subdifferential and the singular subdifferential of the optimal value function presented by~\eqref{estimates_1} and~\eqref{estimates_2}.

\begin{theorem}\label{thm_inclusions} Let $\varphi: X \times Y \rightarrow \overline{\Bbb{R}}$ be a proper convex function and $G: X \rightrightarrows Y$ a convex set-valued map. Then for any $\bar x\in X$ with $\mu(\bar x)\in\mathbb R$, and for any $\bar y \in G(\bar x)$, one has $A\subset B$ and $A^\infty\subset B^\infty$.
\end{theorem}
\begin{proof}  Under the assumptions made, take any $u^*\in A$. Then, there exists a net $\{u^*_\alpha\}\subset X^*$ with $\alpha$ belonging to a directed set $I$ such that $u^*_\alpha\overset{w^*}{\rightarrow} u^*$, and $u^*_\alpha\in A_0$ for all $\alpha\in I$. So, for each $\alpha\in I$, one can find $(x^*_\alpha,y^*_\alpha)\in \partial\varphi(\bar x, \bar y)$ and  $x^*_{1,\alpha}\in D^*G( \bar x, \bar y)(y^*_\alpha)$ with $u^*_\alpha=x^*_\alpha+x^*_{1,\alpha}$. By the inclusion $x^*_{1,\alpha}\in D^*G( \bar x, \bar y)(y^*_\alpha)$ and the definition of coderivative, it holds that $(x^*_{1,\alpha},-y^*_\alpha)\in  N( (\bar x, \bar y); {\rm {gph}}\,G)$. Hence, setting $y^*_{1,\alpha}=-y^*_\alpha$, we have $(x^*_{1,\alpha},y^*_{1,\alpha})\in  N( (\bar x, \bar y); {\rm {gph}}\,G)$ and 
\begin{equation}\label{sect3_expression}\begin{array}{rcl}
	(u^*_\alpha,0) & = & (x^*_\alpha + x^*_{1,\alpha},y^*_\alpha+y^*_{1,\alpha})\\
	& = & (x^*_\alpha,y^*_\alpha)+(x^*_{1,\alpha}+y^*_{1,\alpha})\\
	& \in & \partial\varphi(\bar x, \bar y) + N( (\bar x, \bar y); {\rm {gph}}\,G)\\ & \subset & \overline{\partial\varphi(\bar x, \bar y) + N( (\bar x, \bar y); {\rm {gph}}\,G)}.
	\end{array}
\end{equation} Passing the relations in~\eqref{sect3_expression} to the limit w.r.t. $\alpha\in I$ yields
\begin{equation*}
(u^*,0)\in \overline{\partial\varphi(\bar x, \bar y) + N( (\bar x, \bar y); {\rm {gph}}\,G)}.
\end{equation*} It follows that 
\begin{equation*}
u^*\in 	{\rm pr}_{X^*}\left[\overline{\partial\varphi(\bar x, \bar y) + N( (\bar x, \bar y); {\rm {gph}}\,G)}\cap \big(X^*\times\{0\}\big)\right]=B.
\end{equation*} Thus, we have proved that $A\subset B$.

The inclusion $A^\infty\subset B^\infty$ can be obtained by repeating the above arguments, provided that $A, A_0$, $\partial\varphi(\bar x, \bar y)$, and $B$ are replaced respectively by $A^\infty, A^\infty_0$, $\partial^\infty\varphi(\bar x, \bar y)$, and $B^\infty$.
$\hfill\Box$	
\end{proof}

To prove the next result, we can partially use the construction given by An et al. in~\cite[Section~4]{AnLuanYen_2023}. But, to achieve the aim, we have to represent the Hilbert space $H$ in question as the direct sum of two completely different closed linear subspaces.     

\begin{theorem}\label{thm_solution1} There exist parametric generalized  polyhedral convex optimization problems of the form~\eqref{math_program} where, for some $\bar x\in X$ with $\mu(\bar x)\in\mathbb R$ and for some $\bar y \in M(\bar x)$, one has
\begin{equation}\label{example1} A\subset\partial\mu(\bar x)=B, \ \; A\neq B
	\end{equation}
and 
\begin{equation}\label{example1_sing}	A^\infty\subset\partial^\infty\mu(\bar x)=B^\infty, \ \; A^\infty\neq B^\infty.
\end{equation}
\end{theorem}
\begin{proof} Let $H$ be an infinite-dimensional separable Hilbert space. (For instance, we can choose $H=\ell_2$, the space of real sequences $x=(x_1,x_2,\ldots)$ satisfying the condition $\disp\sum_{k=1}^{\infty}x_k^2<+\infty$, and endow $H$ with the inner product $\langle x,y\rangle = \disp\sum_{k=1}^{\infty} x_ky_k$ and the norm $\|x\|=\langle x,x\rangle^{1/2}$.) Then, $H$ admits an orthonormal vector system $\{e_k\}_{k=1}^\infty$ which is dense in $H$. As it has been shown in~\cite[Example~3.34]{BC_2011}, there exist infinite-dimensional closed linear subspaces $C$ and $D$ of $H$ such that
	\begin{equation}\label{C_D}
		C\cap D=\{0\},\quad C+D\neq H,\quad \overline{C+D}=H.
	\end{equation} By~\eqref{C_D}, there is a vector $z\in H\setminus (C+D)$. Let $X$ be the one-dimensional linear subspace generated by $z$, i.e., $X={\rm span}\{z\}:=\{\alpha z\mid \alpha\in\mathbb R\}$, and $Y:=X^\perp$ with $\Omega^\perp:=\{w\in H\mid \langle w,v\rangle=0\ {\rm for\ every}\ v\in \Omega\}$ denoting the orthogonal subspace to a nonempty subset $\Omega\subset H$. Then, by~\cite[Theorem~12.4]{Rudin_1991} one has $H=X\oplus Y$. So, $H$ can be regarded as the product of the Hilbert spaces $X$ and $Y$; that is, $H=X\oplus Y\equiv X\times Y$. Consider the function  $\varphi: X \times Y \rightarrow \overline{\Bbb{R}}$ defined by setting \begin{equation}\label{varphi_ex1}\varphi(x,y)=\delta((x,y);C^\perp)=\begin{cases} 0 & {\rm if}\  (x,y)\in C^\perp\\
		+\infty & {\rm if}\  (x,y)\notin C^\perp\end{cases}
	\end{equation} 
	and  the set-valued map $G: X \rightrightarrows Y$ defined by setting ${\rm gph}\,G=D^\perp$. 
	
Applying~\cite[Theorem~3.2]{LuanYaoYen_2018}, we can verify that  $\varphi: X \times Y \rightarrow \overline{\Bbb{R}}$ is a proper generalized polyhedral convex function.  The fact that $G: X \rightrightarrows Y$ is a generalized polyhedral convex set-valued map follows from the definition. Observe that $(X\times Y)^*=H^*=H$. Since $C^\perp$ and $D^\perp$ are closed linear subspaces of the Hilbert space $H=X\oplus Y\equiv X\times Y$, from~\eqref{varphi_ex1}, the definition of subdifferential, and Proposition~\ref{prop4.2_An_Yen} it follows that \begin{equation}\label{subdiff_varphi}\partial\varphi(x,y)=\partial^\infty\varphi(x,y)=(C^\perp)^\perp=C\quad \forall(x,y)\in {\rm dom}\,\varphi=C^\perp. 
	\end{equation} In addition, one has
	\begin{equation}\label{normal_gphG} N((x,y); {\rm gph}\,G)=(D^\perp)^\perp=D\quad \forall(x,y)\in {\rm gph}\,G=D^\perp.
	\end{equation} Using~\eqref{C_D}, one can easily prove that $C^\perp\cap D^\perp=\{0\}$. Thus,
	$${\rm dom}\,\varphi\cap {\rm gph}\,G=\{(0,0)\}\subset X\times Y=X\oplus Y=H.$$ So, by the formulas~\eqref{marginalfunction} and~\eqref{solution_map} we get
	$$\mu(x)=\begin{cases} 0 & {\rm if}\ x=0\\  
		+\infty & {\rm if}\ x\neq 0\end{cases}$$ 
	and $$M(x)=\begin{cases} \{0\} & {\rm if}\ x= 0\\  
		\emptyset & {\rm if}\ x\neq 0.\end{cases}$$ 
	Consequently, choosing $\bar x=0$ and $\bar y=0$ gives $\mu(\bar x)\in\mathbb R$, $\bar y \in M(\bar x)$, and 
	\begin{equation*}\partial \mu(\bar x)=\partial^\infty \mu(\bar x)=X^*.
	\end{equation*} 
	(One has $X^*=X$.) In addition, by~\eqref{subdiff_varphi} and~\eqref{normal_gphG} we have $$\partial\varphi(\bar x,\bar y)=\partial^\infty\varphi(\bar x,\bar y)=C$$ and $N((\bar x,\bar y); {\rm gph}\,G)=D.$ Therefore, 
	$$\begin{array}{rcl}	
A&=&\overline{\bigcup\limits_{(x^*,y^*) \in \partial \varphi(\bar x, \bar y)}\big\{x^* + D^*G( \bar x, \bar y)(y^*) \big\}}\\
&=&	\overline{\bigcup\limits_{(x^*,y^*) \in C} \big[x^* + \{u^*\in X^*\mid (u^*,-y^*)\in D\}\big]}\\
&=&	\overline{\bigcup\limits_{(x^*,y^*) \in C} \big\{z^*\in X^* \mid (z^*,0)\in (x^*,y^*)+D\big\}}\\
&=&	\overline{\big\{\alpha z\mid \alpha\in\mathbb R,\ \alpha z+0\in C+D\big\}}\\
&=& \{0\}.
 \end{array}$$ Meanwhile, we have
	$$\begin{array}{rcl}	
	B &=& {\rm pr}_{X^*}\left[\overline{\partial\varphi(\bar x, \bar y) + N( (\bar x, \bar y); {\rm {gph}}\,G)}\cap \big(X^*\times\{0\}\big)\right]\\
	&=&	{\rm pr}_{X^*}\left[\overline{C+D}\cap \big(X^*\times\{0\}\big)\right]\\
	&=& {\rm pr}_{X^*}\left[H\cap \big(X^*\times\{0\}\big)\right]\\
	&=& {\rm pr}_{X^*}\left(X^*\times\{0\}\right)\\
	&=& X^*.
\end{array}$$ Similar computations show that 	$A^\infty=\{0\}$ and $B^\infty=X^*$. 

Summing up all the above, we can assert that the properties in~\eqref{example1} and~\eqref{example1_sing} are valid. $\hfill\Box$	
\end{proof}

\section{Further Investigations on Differential Stability of $(P_x)$}\label{sect_4}

In connection with Proposition~\ref{prop_equalities} and Theorem~\ref{thm_solution1}, the following questions seem to be reasonable and interesting.

\smallskip
\textbf{Question~2.} \textit{The sets $B$ and $B^\infty$ are always closed?}

\smallskip
\textbf{Question~3.} \textit{If $\partial\mu (\bar x)=A$, then $B=A$?} 
	
\smallskip
\textbf{Question~4.} \textit{If $\partial^\infty\mu (\bar x)=A^\infty$, then $B^\infty=A^\infty$?}

\smallskip
\textbf{Question~5.} \textit{Whether the equalities in~\eqref{example1} and~\eqref{example1_sing} hold for any parametric generalized  polyhedral convex optimization problem of the form~\eqref{math_program}, provided that $\bar x\in X$ with $\mu(\bar x)\in\mathbb R$ and $\bar y \in M(\bar x)$?} 

\medskip
The next result answers Question~5 in the affirmative.

\begin{theorem}\label{new-thm} Let $\varphi: X \times Y \rightarrow \overline{\Bbb{R}}$ be a proper generalized polyhedral convex function and $G: X \rightrightarrows Y$ a generalized polyhedral convex set-valued map. Then for any $\bar x\in X$ with $\mu(\bar x)\in\mathbb R$, and for any $\bar y \in M(\bar x)$, one has
		\begin{equation}\label{subdiff_equality1}
			\partial\mu(\bar x) = B
		\end{equation}
	and
		\begin{equation}\label{subdiff_equality2}
		\partial^\infty\mu(\bar x) = B^\infty.
	\end{equation}
	\end{theorem}
	\begin{proof}
		Let $\bar x\in X$ with $\mu(\bar x)\in\mathbb R$ and $\bar y \in M(\bar x)$ be given arbitrarily. By~\eqref{solution_map} we have $\mu(\bar x)= \varphi (\bar x,\bar y)$. Moreover, the inclusion
		\begin{equation}\label{nonempty_intersection}
			(\bar x,\bar y)\in \dom\varphi\cap \gph G
		\end{equation} is valid. Take any vector $\bar x^* \in B$. Then, by~\eqref{B} one has 	
		\begin{equation}\label{xtar_0} (\bar x^*,0) \in\overline{\partial  \varphi(\bar x, \bar y)+ \partial\delta(\cdot; {\rm {gph}}\,G)(\bar x, \bar y)}.\end{equation}
		Since ${\rm {gph}}\,G$ is a nonempty generalized polyhedral convex set, the indicator function $\delta(\cdot;{\rm gph}\,G): X \times Y \rightarrow \overline{\Bbb R}$ is proper generalized polyhedral convex. Using~\eqref{nonempty_intersection} and the equality ${\rm dom}\,\delta(\cdot;{\rm gph}\,G)={\rm gph}\,G$, we can apply the sum rule in Theorem~\ref{thm_sum_rule} to get
		$$\partial\big(\varphi + \delta(\cdot;{\rm gph}\,G)\big) (\bar x,\bar y)=\overline{\partial  \varphi(\bar x, \bar y)+ \partial\delta(\cdot; {\rm {gph}}\,G)(\bar x, \bar y)}.$$
		Then, from~\eqref{xtar_0} it follows that
		\begin{equation}\label{inclusion_basic}
			(\bar x^*, 0) \in \partial\big(\varphi + \delta(\cdot;{\rm gph}\,G)\big) (\bar x,\bar y).
		\end{equation}
		Fix any $x\in X$. By \eqref{inclusion_basic}, 
		for any point $y \in G(x)$ one has
		\begin{align*}
			\varphi(x,y) - \varphi(\bar x, \bar y)\geq \langle \bar x^*, x -\bar x \rangle + \langle 0, y -\bar y \rangle.
		\end{align*}
		So, we get
		\begin{align*}
			\varphi(x,y) - \mu(\bar x)\geq \langle \bar x^*, x -\bar x \rangle, \quad \forall y \in G(x).
		\end{align*}
		This implies that 
		\begin{align*}
			\inf\limits_{y \in G(x)}\varphi(x,y) - \mu(\bar x)\geq \langle \bar x^*, x -\bar x \rangle.
		\end{align*}
		Therefore, $\mu(x) - \mu(\bar x)\geq \langle \bar x^*, x -\bar x \rangle$. Since $x \in X$ was taken arbitrarily, this implies that $\bar x^* \in \partial\mu(\bar x)$. Thus, the inclusion $B \subset \partial\mu(\bar x)$ holds. Since the reverse inclusion is valid by~\eqref{estimates_1}, we have proved the equality in~\eqref{subdiff_equality1}.
		
		Now, to obtain~\eqref{subdiff_equality2}, let us fix any vector $\bar x^* \in B^\infty$. By~\eqref{B-infty} we have
		\begin{equation}\label{xtar_0_infty} (\bar x^*,0) \in\overline{\partial^\infty  \varphi(\bar x, \bar y)+ \partial\delta(\cdot; {\rm {gph}}\,G)(\bar x, \bar y)}.\end{equation}
		Applying Proposition~\ref{prop4.2_An_Yen} yields
		$$\partial^\infty  \varphi(\bar x, \bar y)=N((\bar x, \bar y);{\rm dom}\,\varphi)=\partial\delta ((\bar x, \bar y);{\rm dom}\,\varphi).$$
		Thus, we can equivalently rewrite~\eqref{xtar_0_infty} as
		\begin{equation}\label{xtar_0_infty_1} (\bar x^*,0) \in\overline{\partial\delta ((\bar x, \bar y);{\rm dom}\,\varphi)+ \partial\delta(\cdot; {\rm {gph}}\,G)(\bar x, \bar y)}.\end{equation} Since $\varphi$ is a proper generalized polyhedral convex function, by~\cite[Theorem~3.2]{LuanYaoYen_2018} we see that ${\rm dom}\,\varphi$ is a nonempty generalized polyhedral convex set. Hence, the indicator function $\delta(\cdot;{\rm dom}\,\varphi): X \times Y \rightarrow \overline{\Bbb R}$ is proper generalized polyhedral convex.	In addition, as it has been noted in the first part of this proof, the indicator function $\delta(\cdot;{\rm gph}\,G): X \times Y \rightarrow \overline{\Bbb R}$ is proper generalized polyhedral convex. Therefore, the fact that $(\bar x, \bar y)\in ({\rm dom}\,\varphi )\cap {\rm gph}\,G)$ allows us to apply the sum rule in Theorem~\ref{thm_sum_rule} to obtain 
		\begin{equation}\label{sum_rule_thm3-3} \partial\big(\delta(\cdot;{\rm dom}\,\varphi)\!+\! \delta(\cdot;{\rm gph}\,G)\big) (\bar x,\bar y)\!=\! \overline{\partial\delta ((\bar x, \bar y);{\rm dom}\,\varphi)+ \partial\delta(\cdot; {\rm {gph}}\,G)(\bar x, \bar y)}.\end{equation}
		So, combining~\eqref{xtar_0_infty_1} with~\eqref{sum_rule_thm3-3} gives 
		$$(\bar x^*,0) \in \partial\big(\delta(\cdot;{\rm dom}\,\varphi)+ \delta(\cdot;{\rm gph}\,G)\big) (\bar x,\bar y).$$ Since $\big(\delta(\cdot;{\rm dom}\,\varphi)+ \delta(\cdot;{\rm gph}\,G)\big) (\bar x,\bar y)=0$, the latter implies that 
		$$\big(\delta(\cdot;{\rm dom}\,\varphi)+ \delta(\cdot;{\rm gph}\,G)\big) (x,y)\geq\langle\bar x^*,x-\bar x\rangle$$ for all $(x,y)\in ({\rm dom}\,\varphi)\cap ({\rm gph}\,G)$. This means that 
		\begin{equation}\label{ineq_bar xstar} \langle\bar x^*,x-\bar x\rangle\leq 0\quad\; \forall (x,y)\in ({\rm dom}\,\varphi)\cap ({\rm gph}\,G).\end{equation}
	    For any $x\in {\rm dom}\mu$, since $\mu(x)<+\infty$, by~\eqref{marginalfunction} one can find $y\in G(x)$ such that $(x,y)\in {\rm dom}\,\varphi$. Then one has  $(x,y)\in ({\rm dom}\,\varphi)\cap ({\rm gph}\,G)$. Therefore, from~\eqref{ineq_bar xstar} it follows that $\langle\bar x^*,x-\bar x\rangle\leq 0$. Since this inequality holds for any $x\in {\rm dom}\mu$, we get $\bar x^*\in N(\bar x;{\rm dom}\mu)$. According to Proposition~\ref{prop4.2_An_Yen}, the latter means that $\bar x^*\in \partial^\infty\mu(\bar x)$. Thus, the inclusion $B^\infty\subset \partial^\infty\mu(\bar x)$ has been proved. The reverse inclusion holds by~\eqref{estimates_2}. So, the equality~\eqref{subdiff_equality2} is valid.
		$\hfill\Box$	
\end{proof}

The following theorem solves Questions~2--4 in the affirmative, provided that some standard requirements on $\varphi$, $G$, and $\bar x$ are fulfilled.

\begin{theorem} Under the assumptions of Theorem~\ref{new-thm}, both sets $B$ and $B^\infty$ are closed. Besides, the following assertions are valid:
\begin{itemize}
	\item[{\rm (a)}] If $\partial\mu (\bar x)=A$, then $B=A$.
	\item[{\rm (b)}] If $\partial^\infty\mu (\bar x)=A^\infty$, then $B^\infty=A^\infty$.
\end{itemize}	  
\end{theorem}
\begin{proof} The first claim follows from~\eqref{subdiff_equality1},~\eqref{subdiff_equality2}, and the closedness of the subdifferentials $\partial\mu (\bar x)$ and $\partial^\infty\mu (\bar x)$.
	
	The assertion~(a) is implied by~\eqref{estimates_1} and~\eqref{subdiff_equality1}, while assertion~(b) is immediate from~\eqref{estimates_2} and~\eqref{subdiff_equality2}.
$\hfill\Box$	
\end{proof}

\section{Differential Stability under a Polyhedral Convexity Assumption}\label{sect_5}

It is of interest to know if $\varphi$ is polyhedral convex and $G$ is generalized polyhedral convex, or $\varphi$ is generalized polyhedral convex and $G$ is polyhedral convex, then how the estimates~\eqref{subdiff_outer}--\eqref{sing_inner} look like. The main tool for proving the following two theorems is the sum rule without the closure sign in Theorem~\ref{thm_sum_rule_2}.

\begin{theorem}\label{new-thm_1} Consider the parametric optimization problem ~\eqref{math_program} and suppose that at least one of the following conditions is satisfied:
	\begin{itemize}
		\item [{\rm (i)}] The function $\varphi$ is proper polyhedral convex and the set-valued map $G$ is generalized polyhedral convex;
		\item [{\rm (ii)}] The function $\varphi$ is proper generalized polyhedral convex and the set-valued map $G: X \rightrightarrows Y$ is  polyhedral convex.
	\end{itemize}
	 Then, for any $\bar x\in X$ with $\mu(\bar x)\in\mathbb R$ and for any $\bar y \in M(\bar x)$, one has
	\begin{equation}\label{subdiff_equality1a}
		\partial\mu(\bar x) = A_0=A=B_0=B,
	\end{equation}
	 where $A$, $B$, $A_0$, and $B_0$ are defined respectively in~\eqref{A},~\eqref{B},~\eqref{A_0} and~\eqref{B_0}.
\end{theorem}
\begin{proof} By our assumptions, at least one of the conditions~(i) and~(ii) is satisfied, $\bar x\in X$ is such that $\mu(\bar x)$ is a finite real number, and $\bar y \in M(\bar x)$ is given arbitrarily. We will prove that the relation~\eqref{subdiff_equality1a} is valid.
	
	First, let us show that
	\begin{equation}\label{subdiff_incl1}
		\partial\mu(\bar x)\subset A_0.
	\end{equation}
Since $\bar y \in M(\bar x)$, from~\eqref{solution_map} it follows that $\mu(\bar x)= \varphi (\bar x,\bar y)$. In addition, the inclusion~\eqref{nonempty_intersection} holds. 

Fixing any $\bar x^* \in \partial \mu(\bar x)$, by the definition of subdifferential we have 
	$$\mu(x) -\mu(\bar x) \geq \langle \bar x^*, x-\bar x \rangle$$ for every $x\in X.$
Hence, for each pair $(x,y) \in {\rm gph}\, G$, one gets
	\begin{align*}
		\varphi(x,y) - \varphi(\bar x, \bar y)= \varphi(x,y)-\mu(\bar x) & \geq \mu(x) -\mu(\bar x)\\
		& \geq \langle \bar x^*, x -\bar x \rangle + \langle 0, y -\bar y \rangle.
	\end{align*}
	Clearly, this yields
	\begin{equation*}\
		\big(\varphi+\delta(\cdot;{\rm gph}\,G)\big)(x,y)- \big( \varphi+\delta(\cdot;{\rm gph}\,G)\big)(\bar x,\bar y)\\
		\geq \langle (\bar x^*,0),(x,y)-(\bar x, \bar y)\rangle
	\end{equation*} for all $(x,y)\in X \times Y$. Hence, 
	$(\bar x^*, 0) \in \partial\big(\varphi + \delta(\cdot;{\rm gph}\,G)\big) (\bar x,\bar y).$ 
	
	If the condition~(i) is satisfied then, by choosing $f_1=\varphi$ and $f_2=\delta(\cdot;{\rm gph}\,G)\big)$, from~\eqref{nonempty_intersection} and the sum rule in Theorem~\ref{thm_sum_rule_2} we can deduce that 
	\begin{equation}\label{sum_rule_a} \partial\big(\varphi + \delta(\cdot;{\rm gph}\,G)\big) (\bar x,\bar y)=\partial  \varphi(\bar x, \bar y)+ \partial\delta(\cdot; {\rm {gph}}\,G)(\bar x, \bar y).
\end{equation} Then we have
\begin{equation}\label{inclusion_new1}
		(\bar x^*,0) \in\partial  \varphi(\bar x, \bar y)+ \partial\delta(\cdot; {\rm {gph}}\,G)(\bar x, \bar y).
	\end{equation} By~\eqref{inclusion_new1}, there exists $(x^*,y^*)\in \partial  \varphi(\bar x, \bar y)$ such that 
	$$(\bar x^*-x^*,-y^*) \in \partial\delta(\cdot; {\rm {gph}}\,G)(\bar x, \bar y)=N((\bar x, \bar y);{\rm {gph}}\,G).$$ So, by the definition of coderivative we get the inclusion $\bar x^*-x^*\in D^*G( \bar x, \bar y)(y^*)$, which implies that $\bar x^*\in x^*+ D^*G( \bar x, \bar y)(y^*)$. Thus, remembering that~$A_0$ is defined by~\eqref{A_0} and $\bar x^* \in \partial \mu(\bar x)$ can be chosen arbitrarily, we obtain the inclusion~\eqref{subdiff_incl1}.
	
	If the condition~(ii) is fulfilled then, thanks to~\eqref{nonempty_intersection}, the sum rule in Theorem~\ref{thm_sum_rule_2} works for $f_1:=\delta(\cdot;{\rm gph}\,G)\big) $ and $f_2:=\varphi$. So, we have~\eqref{sum_rule_a}. Hence, repeating the above arguments, we get~\eqref{subdiff_incl1}.
	
    From~\eqref{A},~\eqref{B},~\eqref{A_0} and~\eqref{B_0} we have $A_0\subset A$ and $B_0\subset B$. Thus, by~\eqref{subdiff_incl1},~\eqref{subdiff_equality1}, and Theorem~\ref{thm_inclusions} one gets
  	\begin{equation}\label{two_icls} B=\partial\mu(\bar x)\subset A_0\subset A\subset B.\end{equation}
	In addition, Proposition~\ref{prop_equalities} tells us that $A_0=B_0$. Combining this with~\eqref{two_icls} yields~\eqref{subdiff_equality1a}.
	
	The proof is complete. $\hfill\Box$	
	\end{proof}
	
	\begin{theorem} Under the assumptions of Theorem~\ref{new-thm_1}, for any $\bar x\in X$ with $\mu(\bar x)\in\mathbb R$ and for any $\bar y \in M(\bar x)$, one has
		\begin{equation}\label{subdiff_equality2a}
			\partial^\infty\mu(\bar x) = A_0^\infty=A^\infty=B_0^\infty=B^\infty,
		\end{equation} where  $A^\infty, B^\infty, A_0^\infty$ and $B_0^\infty$ are given respectively by~\eqref{A-infty},~\eqref{B-infty}, ~\eqref{A-infty_0} and~\eqref{B-infty_0}.
	\end{theorem}
	\begin{proof} By the assumptions made, $\bar x\in X$ with $\mu(\bar x)\in\mathbb R$, $\bar y \in M(\bar x)$, and at least one of the conditions~(i) and~(ii) is satisfied. 
		
	Pick a vector $\bar x^* \in \partial^\infty\mu(\bar x)$ and get by Proposition~\ref{prop4.2_An_Yen} that 
		$$\bar x^* \in N(\bar x;{\rm{dom}}\,\mu)=\partial\delta(.;{\rm{dom}}\, \mu)(\bar x).$$
		So, $\langle \bar x^*, x -\bar x \rangle\leq 0$ for every $x\in {\rm{dom}}\, \mu$. Clearly, 		
		if $(x,y) \in \big({\rm{dom}}\, \varphi\big)\cap \big({\rm gph}\, G\big)$, then $x\in {\rm{dom}}\, \mu$. So, we have
		\begin{equation*}
			\begin{split}
			&  \Big(\delta(.;{\rm{dom}}\, \varphi)+\delta(\cdot;{\rm gph}\,G)\Big)(x,y)- \Big(\delta(.;{\rm{dom}}\, \varphi)+\delta(\cdot;{\rm gph}\,G)\Big)(\bar x,\bar y)\\
				& =  0\\
				& \geq 
				\langle (\bar x^*,0),(x,y)-(\bar x, \bar y)\rangle.
			\end{split}
		\end{equation*} If $(x,y) \notin \big({\rm{dom}}\, \varphi\big)\cap \big({\rm gph}\, G\big)$, then $\Big(\delta(.;{\rm{dom}}\, \varphi)+\delta(\cdot;{\rm gph}\,G)\Big)(x,y)=+\infty$. Therefore, 
		\begin{align*} \begin{split}
		&  \Big(\delta(.;{\rm{dom}}\, \varphi)+\delta(\cdot;{\rm gph}\,G)\Big)(x,y)- \Big(\delta(.;{\rm{dom}}\, \varphi)+\delta(\cdot;{\rm gph}\,G)\Big)(\bar x,\bar y)\\
		& \geq  \langle (\bar x^*,0),(x,y)-(\bar x, \bar y)\rangle.
	\end{split}
		\end{align*}
		Thus, the last inequality holds for any $(x,y)\in X \times Y$. It follows that
		\begin{equation}\label{incl-1new} (\bar x^*, 0) \in \partial\big(\delta(.;{\rm{dom}}\, \varphi) + \delta(\cdot;{\rm gph}\,G)\big) (\bar x,\bar y).
		\end{equation} 
		
		If the situation~(i) occurs, then by~\cite[Theorem~3.2]{LuanYaoYen_2018} we can assert that ${\rm{dom}}\, \varphi$ is a nonempty polyhedral convex set; hence $\delta(.;{\rm{dom}}\, \varphi)$ is  a proper polyhedral convex function. In addition, as the set-valued map $G$ is generalized polyhedral convex,  by~\cite[Theorem~3.2]{LuanYaoYen_2018} we know that $\delta(\cdot;{\rm gph}\,G)$ is a generalized polyhedral convex function. From~\eqref{nonempty_intersection} it follows that $$\big({\rm dom}\,\delta(.;{\rm{dom}}\, \varphi)\big)\cap \big({\rm dom}\, \delta(\cdot;{\rm gph}\,G)\big)\neq\emptyset.$$ Therefore, setting 
		$f_1=\delta(.;{\rm{dom}}\, \varphi)$, $f_2=\delta(\cdot;{\rm gph}\,G)\big)$, and applying Theorem~\ref{thm_sum_rule_2} we get
		\begin{equation}\label{sum_rule_1new} \partial\big(\delta(.;{\rm{dom}}\, \varphi) \!+\! \delta(\cdot;{\rm gph}\,G)\big) (\bar x,\bar y)\!=\!\partial  \delta(.;{\rm{dom}}\, \varphi)(\bar x, \bar y)+ \partial\delta(\cdot; {\rm {gph}}\,G)(\bar x, \bar y).
		\end{equation} Consequently, by~\eqref{incl-1new} and Proposition~\ref{prop4.2_An_Yen} we have
		\begin{equation}\label{incl_2new}\begin{array}{rcl}
			(\bar x^*,0) & \in & \partial  \delta(.;{\rm{dom}}\, \varphi)(\bar x, \bar y)+ \partial\delta(\cdot; {\rm {gph}}\,G)(\bar x, \bar y)\\
			& = & \partial^\infty  \varphi(\bar x, \bar y)+N((\bar x, \bar y);{\rm {gph}}\,G).
			\end{array}
		\end{equation} By~\eqref{incl_2new}, we can find $(x^*,y^*)\in \partial^\infty  \varphi(\bar x, \bar y)$ such that 
		$$(\bar x^*-x^*,-y^*)\in N((\bar x, \bar y);{\rm {gph}}\,G).$$ This gives $\bar x^*-x^*\in D^*G( \bar x, \bar y)(y^*)$, which implies that  $\bar x^*\in x^*+ D^*G( \bar x, \bar y)(y^*)$. As~$A^\infty_0$ is defined by~\eqref{A-infty_0} and $\bar x^* \in \partial^\infty \mu(\bar x)$ was taken arbitrarily, we obtain the inclusion
		\begin{equation}\label{subdiff-infty_1new}
		\partial^\infty\mu(\bar x) \subset A_0^\infty.
	  \end{equation}
		
		If the situation~(ii) occurs, then using~\eqref{nonempty_intersection} and the sum rule in Theorem~\ref{thm_sum_rule_2} for $f_1:=\delta(\cdot;{\rm gph}\,G)\big) $ and $f_2:=\delta(.;{\rm{dom}}\, \varphi)$ yields~\eqref{sum_rule_1new}. This gives~\eqref{incl_2new}. Then, by the above arguments we can obtain~\eqref{subdiff-infty_1new}.
		
		Clearly,~\eqref{A-infty},~\eqref{B-infty},~\eqref{A-infty_0} and~\eqref{B-infty_0} imply that $A^\infty_0\subset A^\infty$ and $B^\infty_0\subset B^\infty$. So, from~\eqref{subdiff_equality2},~\eqref{subdiff-infty_1new}, and Theorem~\ref{thm_inclusions} it follows that 
		\begin{equation}\label{two_icls_new} B^\infty=\partial^\infty\mu(\bar x)\subset A^\infty_0\subset A^\infty\subset B^\infty.\end{equation}
		Since the equality $A^\infty_0=B^\infty_0$ is valid by Proposition~\ref{prop_equalities}, we get~\eqref{subdiff_equality2a} from~\eqref{two_icls_new}.
		
		The proof is complete.
		$\hfill\Box$	
		\end{proof}
		
\section{Conclusions}\label{sect_6}
		
		New results on differential stability of infinite-dimensional parametric optimization problems, which are described by proper generalized polyhedral convex functions and generalized polyhedral convex set-valued maps, were obtained. Among other things, relationships between the upper estimates and lower estimates for the subdifferential and the singular subdifferential of the optimal value function given in~\cite{AnLuanYen_2023} were established. We also proved that the lower estimates can be strict, but each upper estimate is an equality. Thus, Question~3 from~\cite{AnLuanYen_2023} has been solved. Besides, we showed that if either the objective function or the constraint set-valued mapping is polyhedral convex, then each subdifferential upper estimate in that paper coincides with the corresponding lower estimate.
		
		Concerning differential stability in convex optimization via generalized polyhedrality, note that Question~2 in~\cite{AnLuanYen_2023} remains open. Its solution may need certain refined sum rules, which are to be found.

\begin{acknowledgements}
This research was supported by the project NCXS02.01/24-25 of the Vietnam Academy of Science and Technology. Duong Thi Viet An and Nguyen Ngoc Luan would like to thank Thai Nguyen University of Sciences and Hanoi National University of Education for creating favorable working conditions. The hospitality of the Vietnam Institute for Advanced Study in Mathematics for our research group in a recent stay is gratefully acknowledged.
\end{acknowledgements}

\section*{Declarations}
\begin{itemize}
	%\item Funding:
	\item[] Conflict of interest/Competing interests: The authors have not disclosed any competing interests.
	%\item Ethics approval
	%\item Consent to participate
	%\item Consent for publication
	%\item[] Availability of data and materials: No data and supplementary material are used in this manuscript.
	%\item Code availability
	%\item Authors' contributions
\end{itemize}

\end{document}